\documentclass[12pt]{amsart}
\usepackage[active]{srcltx}
\usepackage{calc,amssymb,amsthm,amsmath,amscd, eucal,ulem}
\usepackage{alltt}
\usepackage[left=1.35in,top=1.25in,right=1.35in,bottom=1.25in]{geometry}
\RequirePackage[dvipsnames,usenames]{color}

\input{kmacros3.sty}
\input{xy}
\xyoption{all}
\usepackage{tikz}

\numberwithin{equation}{theorem}

\renewcommand{\m}{\mathfrak{m}}

\DeclareMathOperator{\depth}{depth}

\DeclareMathOperator{\chara}{char}
\usepackage{fullpage}

\usepackage{setspace}
\usepackage{hyperref}

\usepackage{enumerate}

\usepackage{graphicx}

\usepackage[all,cmtip]{xy}
\usepackage{verbatim}

\theoremstyle{theorem}

\begin{document}
\title{Maximal Cohen-Macaulay modules over certain Segre products}
\author{Linquan Ma}
\address{Department of Mathematics\\ University of Utah\\ Salt Lake City\\ UT 84112}
\email{lquanma@math.utah.edu}
\thanks{\hspace{1em} The author is partially supported by NSF Grant DMS \#1600198, and NSF CAREER Grant DMS \#1252860/1501102.}
\maketitle

\begin{center}
{\textit{Dedicated to Professor Gennady Lyubeznik on the occasion of his 60th birthday}}
\end{center}

\begin{abstract}
We prove some results on the non-existence of rank one maximal Cohen-Macaulay modules over certain Segre product rings. As an application we show that over these Segre product rings there do not exist maximal Cohen-Macaulay modules with multiplicity less than or equal to the parameter degree of the ring. This disproves a conjecture of Schoutens \cite{SchoutensMCMoverlocaltoricrings}.
\end{abstract}


\section{Introduction}

A finitely generated module $M$ over a Noetherian local ring $(R,\m)$ is called a {\it maximal Cohen-Macaulay module}, abbreviated as MCM or small MCM,\footnote{``Small" here only means $M$ is finitely generated. This notion was introduced in comparison with the notion of {\it big Cohen-Macaulay modules} \cite{HochsterBigCMmodulesandalgebrasandembeddabilityofWittvectors}: modules that are not necessarily finitely generated but are Cohen-Macaulay in a natural sense. However, throughout this article we only deal with finitely generated modules, hence sometimes we simply use MCM instead of small MCM.} if $\depth M=\dim R$. A fundamental and long-standing open question in commutative algebra is that whether every complete local ring $(R,\m)$ admits a finitely generated MCM. This question was known to be true if $\dim R\leq 2$, if $\dim R=3$ and $R$ is $\mathbb{N}$-graded over a field of characteristic $p>0$ (Hartshorne--Peskine-Szpiro--Hochster, see \cite{HochsterBigCMmodulesandalgebrasandembeddabilityofWittvectors}), and in some other special cases (e.g., affine toric rings \cite{HochsterRingsofinvariantsoftorionomialsandpolytopes}). The importance of this question lies in the fact that an affirmative answer will imply Serre's conjecture on positivity of intersection multiplicities \cite{HochsterCohenMacaulaymodules}.

Schoutens \cite{SchoutensMCMoverlocaltoricrings} introduced a stronger version of MCM: in equal characteristic,\footnote{In mixed characteristic the definition of very small MCM is slightly different, we will not work with mixed characteristic in this paper.} $M$ is called a {\it very small} MCM over $(R,\m)$ if $M$ is a finitely generated MCM and $$e(M)\leq \min\{l(R/I)| \mbox{$I$ is generated by a system of parameters}\}$$ where $e(M)$ denotes the multiplicity of $M$ with respect to the maximal ideal $\m$ in $R$ and the right hand side is called the {\it parameter degree} of $R$.

We note that for local rings of dimension $\leq 2$ and for affine toric rings, very small MCM do exist (it is not hard to see that the normalization will do the job in these cases \cite{HochsterRingsofinvariantsoftorionomialsandpolytopes} \cite{HochsterCohenMacaulaymodules}). Some other existence results were established in \cite{SchoutensMCMoverlocaltoricrings}. Moreover, it turns out that under mild conditions, if very small MCM exist in general in characteristic $p>0$, then a reduction to $p>0$ technique will guarantee the existence in characteristic $0$ \cite{SchoutensMCMoverlocaltoricrings} (such reduction process is not known for small MCM). Based on these facts, Schoutens made the following conjecture:
\begin{conjecture}[{\it cf.} Conjecture 1.1 in \cite{SchoutensMCMoverlocaltoricrings}]
\label{conjecture--very small MCM}
Any compete local ring admits a very small MCM.
\end{conjecture}

We will make use of the Segre product of graded rings to obtain a family of counter-examples of Conjecture \ref{conjecture--very small MCM}. The basic idea is as follows: we first prove in Section 2 that under some conditions on the $I$-invariant and multiplicity, very small MCM must have rank one, and then in Section 3 we do computations over a family of Segre product rings to show that they satisfy the conditions in Section 2 and that rank one MCM does not exist over these rings, thus these rings are counter-examples of Conjecture \ref{conjecture--very small MCM}.

\subsection*{Acknowledgement}
I would like to thank Craig Huneke, Anurag Singh, Ehsan Tavanfar, Bernd Ulrich and Uli Walther for many valuable discussions. I would also like to thank the referee for his/her comments on Conjecture \ref{conjecture: small MCM of rank <=N} and for pointing out Remark \ref{remark: non-free small MCM of rank <=N}.

\section{Very small MCM and rank one MCM}

In this section we make some elementary observations on very small MCM. We will see that under certain conditions on $R$, the very small condition will force the module to have rank one. We begin with the following definition.

\begin{definition}
\label{definition--I-invariant}
Let $(R,\m)$ be a local ring of dimension $n$ such that $H_\m^i(R)$ has finite length for every $i<n$. Then we define $$I(R)=\sum_{i=0}^{n-1}\binom{n-1}{i}l(H_\m^i(R))$$ to be the {\it $I$-invariant} of $R$.
\end{definition}

\begin{remark}
\label{remark--remark on finite local cohomology and I-invariant}
It is well-known (for example, see \cite{StuckradandVogelEineVerallgemeinerungderCohenMacaulayRinge} \cite{StuckradandVogelTowardatheoryofBuchsbaumsingularities}) that when $(R,\m)$ is an excellent local domain that is Cohen-Macaulay on the punctured spectrum, $H_\m^i(R)$ has finite length for every $i<n$, and in this case we always have $$0\leq l(R/(\underline{x}))-e(\underline{x}, R)\leq I(R)$$ for every system of parameters $\underline{x}=x_1,\dots,x_n$. 
\end{remark}

As an easy consequence, we observe the following:
\begin{lemma}
\label{lemma--very small implies rank one}
Suppose $(R,\m)$ is an excellent local domain that is Cohen-Macaulay on the punctured spectrum. Suppose also that $R/\m$ is infinite. If $I(R)<e(R)$, then every very small MCM has rank one.
\end{lemma}
\begin{proof}
We may pick $\underline{x}=x_1,\dots,x_n$ a minimal reduction of $\m$ since $R/\m$ is infinite. We know that $e(\underline{x}, R)=e(R)$. If $M$ is a very small MCM over $R$, then by Remark \ref{remark--remark on finite local cohomology and I-invariant}, $$\rank(M)\cdot e(R)=e(M)\leq \min\{l(R/I)| \mbox{$I$ is generated by a system of parameters}\}$$ $$\leq l(R/(\underline{x}))\leq e(R)+I(R)<2\cdot e(R).$$
This shows that $M$ must have rank one.
\end{proof}

When $(R,\m)$ is a normal local domain, every rank one MCM is reflexive, and hence isomorphic to a pure height one ideal of $R$, therefore it corresponds to an element in the class group $Cl(R)$. From this we immediately get:

\begin{proposition}
\label{proposition--non existence of very small over certain UFD}
Suppose $(R,\m)$ is an excellent local UFD of dimension 3, with $R/\m$ infinite. Suppose $R$ is not Cohen-Macaulay and $I(R)<e(R)$, then $R$ does not admit very small MCM.
\end{proposition}
\begin{proof}
It is easy to see that $R$ satisfies the hypothesis of Lemma \ref{lemma--very small implies rank one}. Hence every very small MCM must have rank one. Since $R$ is a UFD, $Cl(R)$ is trivial and thus every rank one reflexive module is isomorphic to $R$. Since $R$ is not Cohen-Macaulay, it follows that rank one MCM does not exists.
\end{proof}

In general, constructing UFDs that are not Cohen-Macaulay turns out to be difficult. In fact, over the complex numbers, complete local UFDs of dimension $\leq 4$ are Cohen-Macaulay \cite{HartshorneOgusOnthefactorialityoflocalrings} (this is not true if we do not assume completeness, see \cite{TavanfarReductiontoUFD}). However, examples of $(R,\m)$ that satisfies Proposition \ref{proposition--non existence of very small over certain UFD} do exist in positive characteristic, see Example 3 and Theorem 2.5 in \cite{SchenzelMarceloNonCMUFDinsmalldimensions} (here the $I$-invariant is 1, so the condition $I(R)<e(R)$ is obvious). Although this example only works in characteristic 2 and is not complete, it already strongly suggests that Conjecture \ref{conjecture--very small MCM} might be false in general. In the next section we will use Segre product to construct complete local rings in arbitrary equal characteristic that do not admit very small MCM.

\section{Computations over Segre products}

\subsection*{Segre product}Let $A$ and $B$ be $\mathbb{N}$-graded rings over a field $A_0=B_0=K$. The {\it Segre product} of $A$ and $B$ is the ring $$R=A\#B=\oplus_{j\geq 0}A_j\otimes_KB_j$$ This ring has a natural grading in which $R_j=A_j\otimes_KB_j$. If $M$ and $N$ are graded modules over $A$ and $B$ respectively, their Segre product is the graded $R$-module $M\#N=\oplus_{j\geq 0}M_j\otimes_KN_j$. If $A$ and $B$ are both normal domains, then $R=A\#B$ is also normal since it is a direct summand of the normal ring $A\otimes_KB$. For reflexive modules $M$ and $N$ over $A$ and $B$ respectively, we have the Kunneth formula for local cohomology \cite{GotoWatanabeOngradedringsI}:
\begin{equation}
\label{equation--local cohomology of segre product}
H_{\m_R}^q(M\#N)=M\#H_{\m_B}^q(N)\oplus H_{\m_A}^q(M)\#N\oplus(\oplus_{i+j=q+1}H_{\m_A}^i(M)\#H_{\m_B}^j(N))
\end{equation}
where $\m_R$, $\m_A$, $\m_B$ denote the homogeneous maximal ideals of the corresponding rings.

\subsection*{Connections with sheaf cohomology} Suppose $A$ and $B$ are normal $\mathbb{N}$-graded $K$-algebras generated over $K$ by degree one forms. Let $X=\Proj A$ and $Y=\Proj B$, then $R$ is the section ring of $X\times Y$ with respect to the very ample line bundle $O_X(1)\boxtimes O_Y(1)$. 

\begin{remark}
\label{remark--class group of the cone}
Now suppose $A$ is a UFD and $B$ is a polynomial ring, both have dimension $\geq 2$. Then $Cl(X)\cong \mathbb{Z}$ and $Y$ is a projective space, so $Cl(X\times Y)=\mathbb{Z}\times \mathbb{Z}$. It follows from Excercise II.6.3 in \cite{Hartshorne} that $Cl(R)=\mathbb{Z}$, and is generated by the class of $$\oplus_{i=0}^\infty H^0(X\times Y, O_X(1+i)\boxtimes O_Y(i))\cong A[1]\#B.$$
\end{remark}

\subsection*{Main result}Now we can state and prove our result on Segre products. We first recall that the {\it $a$-invariant} of an $\mathbb{N}$-graded Cohen-Macaulay ring $R$ is defined to be $a_R=\sup\{j| [H_{\m_R}^{\dim R}(R)]_j\neq 0\}$.

\begin{lemma}[{\it cf.} Example 4.4.13 in \cite{GotoWatanabeOngradedringsI}]
\label{lemma--MCM over Segre products}
Suppose $\dim A\geq 2$ and $\dim B\geq 2$. Then $A[n]\#B$ is MCM over $R=A\#B$ if and only if $a_A<n<-a_B$. In particular there exists $n$ such that $A[n]\#B$ is MCM if and only if $a_A+a_B\leq -2$.
\end{lemma}
\begin{proof}
This follows directly from (\ref{equation--local cohomology of segre product}): the possible non-vanishing lower local cohomology modules of $A[n]\#B$ come from $A[n]\#H_{\m_B}^{\dim B}(B)$ and $H_{\m_A}^{\dim A}(A)[n]\#B$, and these are both zero if and only if $a_A<n<-a_B$.
\end{proof}

\begin{remark}
Although $A[n]\#B$ is not MCM for any $n$ when $a_A+a_B>-2$, $R=A\#B$ might still have rank one MCM in many cases. For example, it follows from Proposition 2.3 (and its proof) in \cite{MaSplittinginintegralextensions} that if $X=\Proj A$ and $Y=\Proj B$ are smooth projective curves, then $R$ has a small MCM of rank one.
\end{remark}

\begin{theorem}
\label{theorem--non existence of rank one MCM}
Let $A=\frac{K[x_1,\dots,x_n]}{(f_1,\dots,f_h)}$ be a graded complete intersection with isolated singularity at $\m_A$. Let $B=K[s,t]$ and $R=A\#B$. Set $d=\deg{f_1}+\deg f_2+\dots+\deg{f_h}$. If $d>n$ and $\dim A=n-h\geq 4$, then $R$ and $\widehat{R}^{\m_R}$ do not have any rank one MCM.
\end{theorem}
\begin{proof}
Since $A$ is a complete intersection that is regular in codimension $3$, $A$ is a UFD by Corollaire XI 3.14 in \cite{GrothendieckSGA2}. By Remark \ref{remark--class group of the cone} we know that $Cl(R)=\mathbb{Z}$ and is generated by the class of $$P=A[1]\#B\cong (x_1s, x_2s,\dots,x_ns).$$

It is easy to check that $Q=A[-1]\#B\cong (x_1s, x_1t)$ represents the inverse of $[P]$ in $Cl(R)$. Since $Cl(R)=\mathbb{Z}$, every rank one MCM must be isomorphic to a symbolic power of $P$ or $Q$, which is $A[n]\#B$ for some $n$ (this is because, under the notation of Remark \ref{remark--class group of the cone}, the class of $A[n]\#B \cong \oplus_{i=0}^\infty H^0(X\times Y, O_X(n+i)\boxtimes O_Y(i))$ is $n$ times the class of $P=A[1]\#B$ in the divisor class group $Cl(R)$, and since it is a rank one reflexive module, it must be isomorphic to $P^{(n)}$). However, since $A$ and $B$ are normal Cohen-Macaulay rings, Lemma \ref{lemma--MCM over Segre products} shows that there is no such $n$ because $a_A+a_B=d-n-2>-2$.

In the case of $\widehat{R}^{\m_R}$, it is enough to observe that $Cl(\widehat{R}^{\m_R})$ is still generated by $[P\widehat{R}^{\m_R}]$. But this follows from a result of Flenner (see 1.5 of \cite{FlennerDivisorclassgroupsofquasihomgeneoussingularities}): if $R$ is $\mathbb{N}$-graded, $R_2$ and $S_3$, then the natural map $Cl(R)\to Cl(\widehat{R}^{\m_R})$ is an isomorphism. Since our ring $R$ is an isolated singularity with $\depth R=\dim R-1=\dim A\geq 4$, $R$ satisfies the hypothesis.
\end{proof}

As an application of Theorem \ref{theorem--non existence of rank one MCM}, we obtain the following which disproves Conjecture \ref{conjecture--very small MCM}:

\begin{corollary}
\label{corollary--main corollary}
Let $R=\frac{K[x_1,\dots,x_n]}{f} \# K[s, t]$ with $K$ an arbitrary infinite field. Suppose $f$ is a degree $n+1$ homogeneous polynomial with isolated singularity at the origin with $n\geq 5$ (e.g., one can take $f=x_1^{n+1}+x_2^{n+1}+\cdots+x_n^{n+1}$ when $\chara(K)\nmid n+1$). Then $\widehat{R}^{\m_R}$ does not admit very small MCM.
\end{corollary}
\begin{proof}
Set $A=\frac{K[x_1,\dots,x_n]}{f}$ and $B=K[s,t]$. Note that $\dim A=n-1$ and $\dim R=n$. Since $\dim_K[R]_i=\dim_K[A]_i\cdot\dim_K[B]_i$, the leading term in the Hilbert polynomial of $R$ is $$\frac{e(A)}{(n-2)!}i^{n-2}\cdot i=\frac{n+1}{(n-2)!}i^{n-1}.$$ Hence we have $$e(\widehat{R}^{\m_R})=e(\m_R,R)=\frac{n+1}{(n-2)!}\times (n-1)!=(n+1)(n-1).$$ On the other hand, using (\ref{equation--local cohomology of segre product}) it is easy to see that $$H_{\m_R}^q(R)=0 \text{ for } q\leq n-2 \text{ and } H_{\m_R}^{n-1}(R)=H_{\m_A}^{n-1}(A)\#B.$$ In particular,
\begin{eqnarray*}
I(\widehat{R}^{\m_R})&=&\sum_{i=0}^{n-1}\binom{n-1}{i}l(H_{\m_R}^i(\widehat{R}^{\m_R}))=l_R(H_{\m_R}^{n-1}(R))\\
&=&\dim_K([H_{\m_A}^{n-1}(A)]_0\otimes[B]_0)+\dim_K([H_{\m_A}^{n-1}(A)]_1\otimes[B]_1)\\
&=&n+2
\end{eqnarray*}
Therefore $I(\widehat{R}^{\m_R})<e(\widehat{R}^{\m_R})$. Since $R$ and hence $\widehat{R}^{\m_R}$ are normal isolated singularities, Lemma \ref{lemma--very small implies rank one} tells us that very small MCM over $\widehat{R}^{\m_R}$ must have rank one. But since $\deg f>n$ and $n\geq 5$, Theorem \ref{theorem--non existence of rank one MCM} implies $\widehat{R}^{\m_R}$ does not have any rank one MCM.
\end{proof}

\begin{remark}
\label{remark--higher rank MCM over segre product}
It is perhaps interesting to point out that in our Theorem \ref{theorem--non existence of rank one MCM} (or Corollary \ref{corollary--main corollary}), $R$ and $\widehat{R}^{\m_R}$ do admit small MCM of higher rank. This follows from Proposition 3.2.2 of \cite{HanesThesis} and the main result of \cite{HerzogUlrichBacklinLinearMCMoverStrictcompleteintersections}.
\end{remark}

Motivated by Theorem \ref{theorem--non existence of rank one MCM} and Remark \ref{remark--higher rank MCM over segre product}, I hazard the following conjecture:

\begin{conjecture}
\label{conjecture: small MCM of rank <=N}
For every integer $N$, there exists a complete local domain $R$ that does not admit small MCM of rank $\leq N$.
\end{conjecture}

\begin{remark}
\label{remark: non-free small MCM of rank <=N}
We point out that, for every $N$, there exists a complete local domain $R$ that does not admit  {\it non-free} MCM of rank $\leq N$. Suppose $R$ is a hypersurface with singular locus of codimension $c\geq 2N+2$ (e.g., one can take $R=\frac{K[[x_1,\dots,x_n]]}{x_1^2+\cdots+x_n^2}$ for $n\geq 2N+3$, where $K$ is a perfect field of $\chara(K)\neq 2$), it then follows from Corollary 2.2 in \cite{BrunsGeneralizedPrincipalIdeal} that any non-free MCM $M$ has rank at least $\frac{c-1}{2}>N$.
\end{remark}

\bibliographystyle{skalpha}
\bibliography{CommonBib}
\end{document}